\documentclass[12pt]{amsart}
\usepackage{a4wide,amssymb}
\usepackage{hyperref}
\usepackage{todonotes}
\date{}

\newtheorem{theorem}{Theorem}
\newtheorem{corollary}[theorem]{Corollary}
\newtheorem{proposition}[theorem]{Proposition}
\newtheorem{question}[theorem]{Question}

\newtheorem{example}[theorem]{Example}
\theoremstyle{definition}

\newcommand{\M}{\mathcal M}
\newcommand{\A}{\mathcal A}

\def\N{\mathbb N}
\def\IN{\mathbb N}

\def\R{\mathbb R}
\def\S{\mathbb S}
\def\T{\mathbb T}
\def\Z{\mathbb Z}
\newcommand{\w}{\omega}

\newcommand{\cov}{\mathrm{cov}}
\newcommand{\cof}{\mathrm{cof}}

\begin{document}
\title{On pseudobounded and premeager paratopological groups}

\author[Taras Banakh]{Taras Banakh}
\address{Taras Banakh: Ivan Franko National University of Lviv (Ukraine) and
Jan Kochanowski University in Kielce (Poland)}
\email{t.o.banakh@gmail.com}
\author[Alex Ravsky]{Alex Ravsky}
\address{Alex Ravsky: Pidstryhach Institute for Applied Problems of Mechanics and Mathematics
National Academy of Sciences of Ukraine}
\email{alexander.ravsky@uni-wuerzburg.de}

\begin{abstract}
Let $G$ be a paratopological group.
Following F.~Lin and S.~Lin, we say that the group $G$ is pseudobounded,
if for any neighborhood $U$ of the identity of $G$,
there exists a natural number $n$ such that $U^n=G$.
The group $G$ is $\omega$-pseudobounded,
if for any neighborhood $U$ of the identity of $G$, the group $G$ is a
union of sets $U^n$, where $n$ is a natural number.
The group $G$ is premeager, if $G\ne N^n$ for any nowhere dense subset $N$ of
$G$ and any positive integer $n$.
In this paper we investigate relations between the above classes of groups and
answer some questions posed by F. Lin, S. Lin, and S\'anchez.
\end{abstract}

\maketitle \baselineskip15pt

A {\em topologized group} $(G,\tau)$ is a group $G$ endowed with a topology $\tau$.
A \emph{left topological group} is a topologized group such that
each left shift $G\to G$, $x\mapsto gx$, $g\in G$, is continuous.
A \emph{semitopological group} $G$ is a topologized group such that
the multiplication map $G\times G\to G$, $(x,y)\mapsto xy$, is separately continuous.
Moreover, if the multiplication is continuous then $G$ is called a {\it paratopological group}.
A paratopological group with the continuous inversion map $G\to G$, $x\mapsto x^{-1}$, is called a
\emph{topological group}. A classical example of a paratopological group failing to be a
topological group is the \emph{Sorgenfrey line} $\mathbb S$, that is the group $\mathbb R$ endowed
with the topology generated by the base consisting of all half-intervals $[a,b)$, $a<b$.

Whereas an investigation of topological groups already is one of fundamental branches of topological
algebra (see, for instance,~\cite{Pon,DikProSto} and~\cite{ArhTka}),
other topologized groups are not so well-investigated and have more variable structure.

Basic properties of semitopological or paratopological groups are described in book \cite{ArhTka}
by Arhangel'skii and Tkachenko, in author's PhD thesis~\cite{Rav3} and papers~\cite{Rav,Rav2}. New
Tkachenko's survey~\cite{Tka} presents recent advances in this area. Let $\w$ be the set of finite
ordinals and $\IN=\w\setminus\{0\}$.

A subset $A$ of a left topological group $G$ is
\begin{itemize}
\item {\it left} (resp. {\it right}) {\it precompact}, if for any neighborhood $U$ of the identity of $G$
there exists a finite subset $F$ of $G$ such that $FU\supseteq A$ (resp. $UF\supseteq A$);
\item {\it precompact}, if $A$ is both left and right precompact;
\item {\it left $\omega$-precompact}, if for any neighborhood $U$ of the identity of $G$
there exists a countable set $F\subseteq G$ such that $FU\supseteq A$;
\item {\it pseudobounded}, if for any neighborhood $U$ of the identity of $G$
there exists  $n\in\IN$ such that $U^n=G$;
\item {\it $\omega$-pseudobounded}, if for any neighborhood $U$ of the identity of $G$,
$G=\bigcup_{n\in\mathbb N} U^n$.
\end{itemize}

Proposition 2.1 from~\cite{Rav2} implies that a paratopological group is left precompact iff it is
right precompact, so iff it is precompact. Moreover, precompact Hausdorff topological groups are
exactly subgroups of compact Hausdorff groups~\cite{Wei}. A left topological group is called {\it
locally left} (resp. {\it $\omega$}-){\it precompact} if it has a left (resp. {\it $\omega$}-){\it
precompact} neighborhood of the identity.

The notion of $\omega$-pseudobounded paratopological groups was introduced by F.~Lin and S.~Lin in
their paper~\cite{LinLin}, generalizing a notion of Azar~\cite{Aza} provided for topological
groups. In ~\cite{LinLin} and also in a subsequent paper~\cite{LinLinSan} with S\'anchez they
investigated basic properties of these groups and asked some related questions. In this paper we
answer some of them.

Similarly to the proof of Theorems 3 and 6 from~\cite{LinLin} we can show that
if $H$ is a normal subgroup of a topologized group $G$ then
$G$ is ($\omega$-)pseudobounded provided both groups $H$ and $G/H$ are ($\omega$-)pseudobounded.


In Problem 2.27 from~\cite{LinLinSan} is asked whether every pseudobounded (para)topological group is
($\omega$)-precompact. The following example provides its negative solution.

\begin{example}
Consider the compact topological group $\mathbb T=\{z\in\mathbb C:|z|=1\}\subset\mathbb C$ endowed
with the operation of multiplication of complex numbers. Let $G$ be $\T^\omega$ endowed with the
topology, generated by the  sup-metric $d(x,y)=\sup_{i\in\omega} |x(i)-y(i)|$ for
$x,y\in\T^\omega$. It is easy to check that $G$ is pseudobounded but not locally
$\omega$-precompact.
\end{example}

On the other hand, the following proposition provides an affirmative solution
to a special case of the above problem.
Recall that a neighborhood $U$ of an identity of a paratopological group $G$ is \emph{invariant},
if $U=g^{-1}Ug$ for each $g\in G$. A paratopological group $G$ is a \emph{SIN-group},
if it has a base at the identity consisting of invariant neighborhoods.


\begin{proposition} Each pseudobounded locally precompact paratopological $SIN$-group
$G$ is precompact.
\end{proposition}
\begin{proof}
Let $U$ be any left precompact neighborhood of the identity $e$ of $G$. Since
the group $G$ is pseudobounded, there exists a natural number $n$ such that $G=U^n$.
Since $G$ is a SIN-group, there exists an invariant neighborhood $V$ of $e$ such that $V^n\subseteq
U$. Since the set $U$ is precompact, there exists a finite subset $F$ of $G$ such that
$FV\supseteq U$. Then $G=U^n\subseteq (FV)^n=F^nV^n\subseteq F^nU$.
\end{proof}


\begin{example} Each real or complex linear topological space is $\omega$-pseudobounded.
Since a first countable topological group (in particular, a normed space) is
$\omega$-precompact iff it is separable, each nonseparable real or complex normed topological
space is $\omega$-pseudobounded but not $\omega$-precompact. For instance,
so is the Banach space $\ell_\infty(X)$ of bounded real-valued functions
on an infinite set $X$, endowed with the supremum norm.
\end{example}




A left topological group $G$ is {\it $2$-pseudocompact}
if $\bigcap_{n\in\omega} \overline{U_n^{-1}}\not=\varnothing$ for each nonincreasing sequence
$(U_n)_{n\in\omega}$ of nonempty open subsets of $G$. By~\cite[Proposition 2.9]{LinLinSan},
each $\omega$-pseudobounded $2$-pseudocompact paratopological group is pseudobounded.
Since each $2$-pseudocompact left topological group is feebly compact and Baire
(by Proposition 3.13 and Lemma 3.7 from~\cite{BR2020}), the following proposition
generalizes this result.

\begin{proposition} Each $\omega$-pseudobounded feebly compact Baire left topological group $G$ is
pseudobounded.
\end{proposition}
\begin{proof} Let $U$ be any neighborhood of the identity of $G$. Since the group $G$ is
$\omega$-pseu\-do\-bo\-un\-ded, $G=\bigcup_{n\in\N} U^{-n}$. Since the space $G$ is Baire, there
exists $n\in\N$ such that $\overline{U^{-n}}$ contains a nonempty open set $V$. Then
$U^{-n-1}=U^{-n}U^{-1}\supseteq\overline{U^{-n}}\supseteq V$. Pick any point $y\in V$. Since
$G=\bigcup_{n\in\N} U^{n}$, there exists $m\in\N$ such that $y\in U^{m-1}$. Then
$W:=y^{-1}V\subseteq U^{-m+1}U^{-n-1}=U^{-m-n}$. Suppose for a contradiction that $G\ne U^k$ for
every $k\in\IN$. Taking into account that $\overline{U^{-k}}\subseteq U^{-k}U^{-1}$, we conclude
that $\overline{U^{-k}}\ne G$ for every $k\in\IN$. Since the group $G$ is feebly compact, there
exists a point $x\in \bigcap_{k\in\N}\overline{G\setminus \overline{U^{-k}}}$. Since the group $G$
is $\omega$-pseudobounded, there exists $l\in\IN$ such that $x\in W^l\subseteq U^{-l(m+n)}$. But
then $W^l$ is a neighborhood of $x$, disjoint from $G\setminus \overline{U^{-l(m+n)}}$, a
contradiction.
\end{proof}

The following proposition answers Question 6 from~\cite{LinLin}.

\begin{proposition} Let $G$ be a pseudobounded left topological group and $d$ be any
left-invariant quasi-pseudometric generating the topology of $G$. Then $d$ is bounded on $G$.
\end{proposition}
\begin{proof}
Since $G$ is pseudobounded, for the neighborhood $U=\{x\in G:d(e,x)<1\}$ of the identity $e$ in
$G$, there exists a  number $n\in\IN$ such that $G=U^n$. Now let $y,z$ be any elements of $G$.
There exist elements $x_1,\dots, x_n\in G$ such that $y^{-1}z=x_1\cdots x_n$. Then
$$d(y,z)=d(e,y^{-1}z)=d(e, x_1\cdots x_n)\le$$ $$d(e, x_1)+d(x_1,x_1x_2)+\dots +d(x_1\cdots
x_{n-1},x_1\cdots x_{n})=$$ $$d(e, x_1)+d(e,x_2)+\dots +d(e,x_{n})<n.$$
\end{proof}



Following F.~Lin and S.~Lin~\cite{LinLin}, we call a
left topological group $G$ \emph{premeager}, if $G\ne N^n$ for any nowhere dense subset $N$ of
$G$ and any $n\in\IN$.

A \emph{Lusin space} is an uncountable crowded $T_1$ space containing no uncountable nowhere dense subsets.
A space $X$ is {\em crowded} if every nonempty open set in $X$ is infinite. Clearly,
each Lusin left topological group is premeger. A trivial example of a Lusin space is any
uncountable space $X$ endowed with the $T_1$-topology $\{\emptyset\}\cup\{X\setminus A:A$ is
finite$\}$. On the other hand, the existence of a Hausdorff Lusin space is independent of the
axioms of ZFC: Lusin~\cite{Lus} showed that such a space exists under Continuum Hypothesis, and
Kunen~\cite{Kun} showed that there are no Hausdorff Lusin spaces under Martin's Axiom and the negation of
Continuum Hypothesis.


In order to construct a nonpremeager paratopological group we introduce the
following definition.
A subset of a space $X$ is \emph{meager}, if it is contained in a countable
union of nowhere dense subsets of $X$. By $\M_X$ we denote the family of
all meager subsets of $X$. Let
\begin{itemize}
\item $\cov(\M_X)=\min\{|\A|: \A\subseteq \M_X \;\; (\bigcup\A=X) \}$ and
\item $\cof(\M_X)=\min\{|\A|: \A\subseteq \M_X \;\; \forall B\in\M_X\; \exists A\in \A\; (B\subseteq A)\}$.
\end{itemize}
The family $\A$ in the latter definition is called a~\emph{cofinal} in $\M_X$.
The family $\M_\R$ will be denoted by $\M$.
A space is called \emph{Polish}, if it is homeomorphic to a separable complete metric space.
By Theorem 15.6 in~\cite{Kech}, for any crowded Polish space $X$  we have
$\cof(\M_X)=\cof(\M)$ and $\cof(\M_X)=\cof(\M)$.

A left topological group $G$ \emph{meagerly divisible} if for every meager subset
$M$ of $G$ and every nonzero integer number $n$, the set $\{x\in G:x^n\in M\}$ is meager.

\begin{proposition}\label{pOpenMD} A paratopological group $G$ is meagerly divisible provided
for each positive integer $n$, the power map $p_n:G\to G$, $x\mapsto x^n$, is open.
\end{proposition}

\begin{proof}
Since $G$ is a topological group, for every $n\in\IN$ the openness of the map $p_n:G\to G$ implies
the openness of the map $p_{-n}:G\to G$, $p_{-n}:x\mapsto x^{-n}$. Let $N$ be any closed nowhere
dense subset of $G$ and $n$ be any nonzero integer. Since the map $p_n$ is continuous, the
preimage $p_n^{-1}(N)$ is closed. If $p_n^{-1}(N)$ contains a nonempty open subset $U$ of $G$ then
$p_n(U)$ is a nonempty open subset of $N$, that is impossible.

Let $M$ be any meager subset of $G$. There exist a countable family $\mathcal A$ of nowhere
dense closed subsets of $G$ such that $M\subset\bigcup\mathcal A$. Then $p_n^{-1}(M)\subseteq\bigcup
\{p_n^{-1}(N):N\in\mathcal A\}$ and each set $p_n^{-1}(N)$ is nowhere dense.
\end{proof}

\begin{corollary}\label{cRMD} The paratopological groups $\R$, $\S$, $\R/\Z$, and $\S/\Z$
are meagerly divisible.
\end{corollary}
\begin{proof}
The topological group $\R$ and $\R/\Z$ are meagerly divisible by Proposition~\ref{pOpenMD}. The
Sorgenfrey line $\S$ (resp. $\mathbb S/\Z$) has a common $\pi$-base with the topological group $\R$
(resp. $\R/\Z$), so $\M_{\S}$ (resp. $\M_{\S/\Z}$) equals $\M$ and the group $\S$ (resp. $\S/\Z$)
is meagerly divisible.
\end{proof}

We recall that a space is \emph{analytic} if it is a continuous image of a Polish space.
The Open Mapping Principle (see, for instance, Corollary 3.10 in~\cite{BGJS})
states that any continuous surjective homomorphism from
an analytic topological group to a Polish topological group is open.

\begin{proposition} Each divisible Abelian Polish topological group $G$ is meagerly divisible.
\end{proposition}
\begin{proof}
Let $n$ be any nonzero integer and $p_n:G\to G$, $x\mapsto nx$, be the power map.
We claim the the map $p_n$ is open. Indeed, since $G$ is Abelian, $p_n$ is a homomorphism.
Since $G$ is divisible, $p_n$ is surjective. By the Open Mapping Principle,
$p_n$ is open. By Proposition~\ref{pOpenMD}, $G$ is meagerly divisible.
\end{proof}

\begin{proposition} If an Abelian Polish topological group $G$ is meagerly divisible, then
for each nonzero integer $n$, the power map $p_n:G\to G$, $x\mapsto nx$, is open.
\end{proposition}
\begin{proof} Let $n$ be any nonzero integer. Since $p_n$ is a homomorphism, it suffices to
show that for each neighborhood $U$ of the identity, $p_n(U)$ is a neighborhood of the identity.
Pick an open neighborhood $V$ of the identity such that $V-V\subseteq U$.
Since the group $G$ is Polish, the set $V$ is nonmeager.
The group $G$ is meagerly divisible, the set $p_n(V)$ is nonmeager.
Since the map $p_n$ is continuous, the space $p_n(V)$ is analytic.
Piccard-Pettis' Theorem implies (see, for instance, Corollary 3.9 in~\cite{BGJS})
that $p_n(V)-p_n(V)$ is a neighborhood of the identity of $G$.
It remains to note that $p_n(U)\supseteq p_n(V)-p_n(V)$.
\end{proof}

\begin{proposition}\label{pPremExists} Let $H$ be a meagerly divisible
Abelian semitopological group such that $\cov(\mathcal{M}_H)=\cof(\mathcal{M}_H)=\kappa>\omega$.
Then $H$ contains a free group $G$ of size $\kappa$ such that every meager subset $M\subseteq G$
of $H$ has cardinality less than $\kappa$ and hence does not generate $G$. If $H$ is $T_1$ (and
$\kappa=\w_1$), then $G$ is crowded (and Lusin).
\end{proposition}

\begin{proof}
Let $\{S_\alpha:\alpha<\kappa\}$ be a cofinal family in
$\mathcal{M}_H$ such that $S_\alpha\ni e$ for each $\alpha$.
Since the group $H$ is meagerly divisible, it contains a nonperiodic element $x_0$.
For each nonzero integer $n$ let $p_n:H\to H$, $x\mapsto nx$, be the power map.
Since $\kappa=\cov(\mathcal{M}_H)$ and $H$ is meagerly divisible,
we can inductively choose for each nonzero ordinal $\beta<\kappa$ a point
$$x_\beta\in H\setminus
\bigcup\{p_n^{-1}(S_\alpha+G_\beta):n\in\mathbb Z\setminus\{0\},\, \alpha<\beta\},$$
where $G_\gamma=\langle\{x_\alpha:\alpha<\gamma\}\rangle$ for each $\gamma\le\kappa$.
Put $G=G_\kappa$.

We claim that $G$ is a free Abelian group over the set $\{x_\alpha:\alpha<\kappa\}$. Indeed,
suppose for a contradiction that there exist a finite subset $F$ of $\kappa$ and a map
$f:F\to\mathbb Z\setminus \{0\}$ such that $\sum_{\alpha\in F} f(\alpha)x_\alpha=0$. Put
$\gamma=\max F$. Then $x_\gamma\in p_{f(\gamma)}^{-1}(G_\gamma)$, a contradiction.


Let $M\subseteq G$ be a meager subset of the space $H$.  Then $M$
is contained in the set $S_\beta$ for some ordinal
$\beta<\kappa$. For every $x\in M\subseteq G$, there exist a finite
subset $F$ of $\kappa$ and a function $f:F\to\mathbb Z\setminus \{0\}$ such that $x=\sum_{\alpha\in F}
f(\alpha)x_\alpha$. Put $\gamma=\max F$. Then
$x_\gamma\in p_{f(\gamma)}^{-1}(S_\beta+G_\gamma)$ and hence $\gamma\le\beta$,
which implies that $x\in G_{\beta+1}$. Thus
$M\subseteq S_\beta \cap G\subseteq G_{\beta+1}$ and $|M|\le|G_{\beta+1}|<\kappa$.

If $H$ is a $T_1$ space,  then $H$ is not discrete (otherwise $\mathcal M_H=\emptyset$) and
$\cof(\mathcal M_H)=1$ and the group $G$ is crowded (otherwise $G$ is discrete, meager in $H$ and
hence $|G|<\kappa=|G|$).
\end{proof}

\begin{example}\label{ePrem} The assumption
$\cov(\mathcal{M})=\cof(\mathcal{M})=\kappa>\omega$ is consistent (for instance, $MA_{countable}$
implies $\cov(\mathcal{M})=\cof(\mathcal{M})=\mathfrak c$). Since, by Corollary~\ref{cRMD}, the
topological group $\R$ is meagerly divisible, by Proposition~\ref{pPremExists}, it contains a
premeager crowded group $G$ of size $\kappa$.
\end{example}

\begin{corollary} The existence of a Lusin Hausdorff  (para)topological group is independent
on ZFC.
\end{corollary}
\begin{proof}
By Corollary~\ref{cRMD}, the topological group $\R$ is meagerly divisible.
By Proposition~\ref{pPremExists}, under the consistent assumption $\cof(\mathcal{M})=\omega_1$,
$\R$ contains a Lusin group $G$ of size $\omega_1$.

On the other hand, by \cite{Kun}, under Martin's Axiom and the negation of Continuum Hypothesis,
there are no Hausdorff Lusin spaces.
\end{proof}

It turns out that it is independent on ZFC
whether each pseudobounded Lusin Hausdorff paratopological group is a topological group, see
Question 4 from~\cite{LinLin}.
Indeed, under Martin's Axiom and the negation of Continuum Hypothesis, each
Lusin paratopological group has an isolated point and so it is a topological group.
On the other hand, under $\cof(\mathcal{M})=\omega_1$ we have the following
example (which is also a counterexample for Questions 1 and 3 from~\cite{LinLin}).

\begin{example} Assume $\cof(\mathcal{M})=\omega_1$.
Since, by Corollary~\ref{cRMD}, the pa\-ra\-to\-po\-lo\-gi\-cal group $\S/\Z$ is meagerly divisible,
by Proposition~\ref{pPremExists}, it contains a Lusin free group $G$ of size $\omega_1$.
It follows that $G$ is nondiscrete and so not a topological group.

We claim that that the group $G$ is pseudobounded. Indeed, let $U$ be any open neighborhood
of the identity of the group $G$. Let $\tau$ be a topology on $G$ inherited from $\R/\Z$.
Since the group $\R/\Z$ is compact, the group $(G,\tau)$ is precompact, see~\cite{Wei}.
Since the set $U$ has nonempty interior in $(G,\tau)$, there exists a finite subset $F$
of $G$ such that $G=F+U$. Since $F$ is finite, it suffices to show that $F$
is contained in a subsemigroup $S$ of $G$, generated by $U$. We claim that $S=G$.
Indeed, following~\cite[Section 5.1]{BR2020}, consider a (not necessarily Hausdorff)
paratopological group $G_S$ whose topology consists of the sets $A+S$ where $A\subseteq G$
is any subset. Since the topology of the group $G_S$ is weaker than the topology of $G$ and $G$ is
precompact, the group $G_S$ is precompact too, and so by Proposition 5.8
from~\cite{BR2020}, $S$ is a subgroup of $G$. 
Since $G$ is nondiscrete, $U$ contains
cosets $a+\Z$ for arbitrarily small positive numbers $a$, so $S$ is dense in $G$. Since $U$
is open in $G$, $S$ is open in $G$. So $S$ is an open dense subgroup of $G$, that implies $S=G$.
\end{example}

\begin{question}\label{qPremExistsZFC} Whether there exists under ZFC a nondiscrete premeager
(regular) Hausdorff (pa\-ra)\-to\-po\-lo\-gi\-cal group?
\end{question}

%

The class of nonpremeager topological groups is rather wide. For instance,
according to exercises from~\cite[Section 13]{PB}, a topological group $G$
contains a closed nowhere dense (and left Haar null in all cases below but the first)
subset $N$ such that $NN^{-1}=G$
in the following cases:
\begin{itemize}
\item $G$ is nondiscrete Polish Abelian;
\item $G$ is nondiscrete locally compact $T_1$;
\item $G$ is metrizable and contains a closed connected Lie subgroup;
\item $G$ is a $T_1$ real linear topological space, which is metrizable or locally convex;
\item $G$ is the group of homeomorphisms of the Hilbert cube, endowed with the compact-open topology.
\end{itemize}


Moreover, according to~\cite[Question 13.6]{PB} it is not known even whether
there exists under ZFC a $T_1$ topological group that cannot be 
generated by its meager subset.



A special case of a nowhere dense subset of a $T_1$ crowded space is a discrete subset.
There is a known problem when a $T_1$ (para)topological group $G$ can be (topologically) generated
by its discrete subset $S$ such that $S\cup\{e\}$ is closed. Namely,
such sets for topological groups were considered by Hofmann and Morris in~\cite{HM1990}
and by Tkachenko in~\cite{Tka97}.
Fun\-da\-men\-tal results were obtained by Comfort et al. in~\cite{CMRS1998}
and Dikranjan et al. in~\cite{DTT1999} and in~\cite{DTT2000}.
In~\cite{LRT} Lin et al. extended this research to paratopological groups.


A sample result is Theorem 2.2 from~\cite{CMRS1998} stating that
each countable Hausdorff topological group is generated by a closed discrete set.
Protasov and Banakh in Theorem 13.3 of ~\cite{PB}
generalized this and Guran's~\cite{G2003} results to Hausdorff left topological groups.

\end{document}